\documentclass[12pt]{amsart}
\usepackage{amsmath,amssymb,amsthm,amscd,amsopn,amsfonts,amsxtra} 
\usepackage{latexsym,array}
\usepackage[mathscr]{eucal}
\usepackage{graphicx,psfrag}
\usepackage{epsfig}
\usepackage{color}
\newtheorem{theorem}[subsection]{Theorem} 
\newtheorem{lemma}[subsection]{Lemma}
\newtheorem{corollary}[subsection]{Corollary}

\newtheorem{remark}[subsection]{Remark}

\def\al{\alpha}
\def\de{\delta}
\def\eps{\epsilon}

\def\vphi{\varphi}
\def\De{\Delta}

\def\supp{\mathrm{supp}}

\def\iy{\infty}
\def\pa{\partial} 

\newcommand{\cal}{\mathcal}

\newcommand{\la}{\langle}
\newcommand{\ra}{\rangle}
\newcommand{\nd}{\noindent}

\newcommand{\vs}{\vspace}
\newcommand{\n}{\newline}

\newcommand{\hB}{\hfill$\Box$}
\newcommand{\hr}{\hookrightarrow}

\newcommand{\R}{\mathbb{R}}

\def\sideremark#1{\ifvmode\leavevmode\fi\vadjust{\vbox to0pt{\vss
 \hbox to 0pt{\hskip\hsize\hskip1em
\vbox{\hsize2cm\tiny\raggedright\pretolerance10000
 \noindent #1\hfill}\hss}\vbox to8pt{\vfil}\vss}}}%

{\end{list}}

\begin{document}
\title[Time decay for Schr\"odinger equation] 
{Time decay for  Schr\"odinger equation with rough potentials} 
\author{Shijun Zheng}
\address[]{Department of Mathematical Sciences\\
                   Georgia Southern University\\
                   Statesboro, GA 30460-8093 \\
                   USA} 
\email{szheng@georgiasouthern.edu}
 \urladdr[]{http://math.georgiasouthern.edu/~szheng}

\date{September 30, 2007}

\begin{abstract}
We obtain certain time decay and regularity estimates for 
3D Schr\"odinger equation with a potential in the Kato class by using Besov spaces 
associated with Schr\"odinger operators.
\end{abstract}

\maketitle 

\section{Introduction}\label{S1}  

The Schr\"odinger equation $iu_t=-\De u$ 
describes the waves of a free particle in a non-relativistic setting.  
It is physically important to consider a perturbed dispersive 
system in the presence of interaction between fields. 

Let $H=-\De+V$, where $\De$ is the Laplacian and $V$ is a real-valued function on $\R^n$. 
In this note we are concerned with the time decay of  Schr\"odinger equation with a potential 
\begin{align*}
&iu_t=H u, \\
&u(x,0)=u_0 ,\notag
\end{align*}
where the solution is given by $u(x,t)=e^{-itH}u_0$.
For simple exposition we consider the three dimensional case for $V$ in the Kato class \cite{Si82,JN94}.
Recall that $V$ is said to be in the {\em Kato class} $K_n$, $n\ge 3$ provided
\begin{align*}
\lim_{\de\to 0+} \sup_{x\in \R^n}\int_{|x-y|<\de} \frac{|V(y) |}{|x-y|^{n-2}}dy=0 . 
\end{align*}
Throughout this article we assume that $V=V_+-V_-$, $V_\pm\ge 0$ so that 
$V_+\in K_{n,loc}$ and $V_-\in K_n$, where 
$V\in K_{n,loc}$ if and only if $V\chi_{B}\in K_n$ for any characteristic function $\chi_B$ of the balls 
$B$ centered at $0$ 
in $\R^n$. 

We seek to find minimal smoothness condition on the initial data $u_0=f$ so that $u(x,t)$
has certain global time decay and regularity estimates. The idea is to combine the results of Jensen-Nakamura and
Rodnianski-Schlag \cite{JN94,RS04}
for short and long time decay by using 
Besov space method. 

In \cite{BS93,JN94,DP05,OZ06,Z06a}  several authors introduced and studied the 
Besov spaces and Triebel-Lizorkin spaces associated with $H$.  
Let $\{\varphi_j\}_{j=0}^\iy\subset C_0^\infty ({\R}) $ be a dyadic system satisfying 
\begin{enumerate} 
\item[(i)]  $\supp \,\vphi_0 \subset \{ x: |x|\le 1\}$, $\supp\, \varphi_j
\subset \{ x: 2^{j-2}\le |x|\le 2^{j}\} $, $j\ge 1$,
\item[(ii)] $|\vphi_j^{(k)}(x)|\le c_k 2^{-kj}\, ,    \qquad \forall k\ge 0, j\ge 0$,
\item[(iii)]  $\displaystyle{\sum_{j=0}^\infty |\varphi_j(x)| = 1, \quad   \forall x .}$
\end{enumerate} 
Let $\alpha\in \R$,  $1\le p\le \infty, 1\le q\le \infty$. The (inhomogeneous)  
\emph{Besov space} \emph{associated with} $H$, denoted by ${B}_p^{\alpha,q}(H)$, %
is defined to be the completion of $\mathcal{S}(\R^n)$, the Schwartz class, with respect to  
the norm 
$$ \Vert f\Vert_{{B}_p^{\alpha,q}(H) } = 
\big(\sum_{j=0}^{\infty}
 2^{j\alpha q} \Vert \vphi_j(H)f \Vert_{L^p}^q \big)^{1/q}\, .
$$

Similarly, the (inhomogeneous) {\em Triebel-Lizorkin space associated with $H$}, denoted by ${F}_p^{\alpha,q}(H)$, $\alpha\in \R$,  $1\le p< \infty, 
1\le q\le \infty$ is defined by the norm
\begin{equation*}\label{eq:F-norm}
\Vert f\Vert_{{F}_p^{\alpha,q}(H)} 
=\Vert \big(\sum_{j=0}^{\infty} 2^{j\alpha q} \vert \vphi_j(H)f \vert^q\big)^{1/q}\Vert_{L^p} \, .
\end{equation*} 


The main result is the following theorem. 
Let $\Vert V\Vert_K$ denote the 
Kato norm 
$$ \Vert V\Vert_{K} := \sup_{x\in \R^3}\int_{\R^3} \frac{|V(y) |}{|x-y|}dy . $$
 Let $\beta: =\beta(p)=n\vert \frac{1}{p}-\frac{1}{2}\vert$ be the critical exponent.
\begin{theorem}\label{th:global-V} Let $1\le p\le 2$.  Suppose 
$V\in K_n$, $n=3$ so that $\Vert V\Vert_K<4\pi$ and
\begin{align}\label{e:rollnik-4pi-square}
\int_{\R^6} \frac{|V(x)| \,|V(y) |}{|x-y|^2}dxdy <(4\pi)^2 . \end{align}
The following statements hold.
a) If $0<t\le 1$,  then 
\begin{equation}\label{e:e-Lp-tB-short-time}
\Vert e^{-itH}f \Vert_{p'}\lesssim \Vert f\Vert_{p'}  +  t^\beta \Vert f \Vert_{B^{\beta,1}_{p'}(H)} .
\end{equation}
b) If in addition, $|\pa_x^\al V(x)|\le c_\al$, $|\al|\le 2n$, $n=3$, 
then for all $t>0$
\begin{equation}\label{e:u-B-Lp}
\Vert  e^{-itH}f\Vert_{L^{p'}} \lesssim \la t \ra^{-n(\frac{1}{p}-\frac{1}{2})} \Vert  f\Vert_{B^{2\beta,1}_p(H)} ,
\end{equation}
where $p'=p/(p-1)$ is the conjugate of $p$ and $\langle t\rangle=(1+t^2)^{1/2}$.
\end{theorem} 
\begin{remark} The short time estimate in (\ref{e:e-Lp-tB-short-time}) is an improvement upon 
 \cite{JN94}
since we only demand smoothness order being $\beta$ rather than $2\beta$. 
\end{remark}
It is well known that if $V$ satisfies (\ref{e:rollnik-4pi-square}), then $\sigma(H)=\sigma(H_{ac})=[0,\iy)$.
Note that by Hardy-Littlewood-Sobolev inequality, 
$V\in L^{3/2}$ implies the finiteness of the L.H.S of (\ref{e:rollnik-4pi-square}).  Moreover, 
$V\in L^{3/2+}\cap L^{3/2-}$ implies $\Vert V\Vert_K<\iy$
\cite[Lemma 4.3]{DP05}. 
In particular, if $\Vert V\Vert_{L^{3/2+}\cap L^{3/2-}}$ is
sufficiently small, then the conditions of Theorem \ref{th:global-V} (a) 
are satisfied. 

The proof of the main theorem is a careful modification of that of the one dimensional result for a special potential in \cite{OZ06}. 
For short time we obtain (\ref{e:e-Lp-tB-short-time}) by modifying the proof of \cite[Theorem 4.6]{JN94}. 
The long time estimates simply follows from the $L^p\to L^{p'}$ estimates for $e^{-itH} $, 
$1\le p\le 2$, a result of \cite[Theorem 2.6]{RS04}, 
and the embedding $B^{\eps,q}_p(H)\hr L^{p}$, $\eps>0$, $1\le p,q\le \iy$.



Note that from the definitions of $B(H)$ and $F(H)$ spaces we have
\begin{align}\label{e:B-F-B-emb}
B^{\al,\min(p,q)}_p(H)\hr  F^{\al,q}_p(H)\hr B^{\al,\max(p,q)}_p(H) 
\end{align}
for $1\le p<\iy$, $1\le q\le\infty$, where $\hr$ means continuous embedding. 


\section{Proof of Theorem \ref{th:global-V}} 

The following lemma is proved in \cite[Theorem 2, Remark 2.2]{JN94}.
\begin{lemma}\label{l:phi-et-beta} (\cite{JN94}) Let $1\le p\le \iy$. Suppose $V\in K_n$, $n=3$ and $\phi\in C^\infty_0(\R)$. Then 
there exists a constant $c>0$ independent of $\theta\in(0, 1]$ so that
\begin{equation*}
\Vert \phi(\theta H)e^{-it\theta H} f\Vert_p\le c\la t\ra^\beta\Vert f\Vert_p\, . 
\end{equation*}
\end{lemma}

\begin{remark} We can also give a simple proof of this lemma based on the fact that the heat kernel of $H$ satisfies an upper Gaussian
bound in short time.  The interested reader is referred to \cite{Z06a} and \cite{JN94,Si82}.
\end{remark}

The long time decay has been studied quite extensively under a variety of conditions on $V$
\cite{JSS,RS04,Sch05a,V06a,Y95}. 
The following $L^p\to L^{p'}$ estimates follow via interpolation 
between the $L^2$ conservation 
and the $L^1\to L^\infty$ estimate  for $e^{-itH}$ 
that was proved in \cite[Theorem 2.6]{RS04}.
\begin{lemma}\label{l:disp-RS} Let $1\le p\le 2$.  Suppose $\Vert V\Vert_K<4\pi$ and
\begin{align*}
\int_{\R^6} \frac{|V(x)| \,|V(y) |}{|x-y|^2}dxdy <(4\pi)^2 . \end{align*}
 Then $ \Vert e^{-itH}f\Vert_{L^{p'}} \lesssim  |t|^{-n(\frac{1}{p}-\frac{1}{2})} \Vert f\Vert_{L^p} $.
\end{lemma}




\subsection{Proof of Theorem \ref{th:global-V}} (a) Let $0<t\le 1$.  Let $\{\varphi_j\}_{j=0}^\iy$ 
be a smooth dyadic system as given in Section 1. 
For $f\in \cal{S}$ we write
\begin{align}\label{e:eitHf-sum}
e^{-itH}f =\sum_{2^jt\le 1} \vphi_j(H)e^{-itH}f  + \sum_{2^j t>1} \vphi_j(H)e^{-itH}f \,.
\end{align}
According to Lemma \ref{l:phi-et-beta}, if $j\ge j_t:=[-\log_2 t]+1$,
\begin{align*}
 \Vert\vphi_j(H)e^{-itH}f\Vert_{p'} \le c \la t2^{j}\ra^\beta \Vert {\vphi}_j (H) f \Vert_{p'}
\end{align*}
where we noted that $\vphi_j(H)=\psi_j(H)\vphi_j(H)$, $\psi_j=\psi(2^{-j}x)$  
if taking $\psi\in C^\iy_0$ so that $\psi(x)\equiv 1$ on $[-1,-\frac14]\cup [\frac{1}{4},1]$. 
It follows that
\begin{align*}
 &\sum_{2^j t>1} \Vert \vphi_j(H)e^{-itH}f \Vert_{p'}
\le c \, t^\beta \sum_{2^j t>1}  2^{j\beta} \Vert  {\vphi}_j (H) f \Vert_{p'} \,.
\end{align*}
For the first term in the R.H.S. of (\ref{e:eitHf-sum}), 
 similarly we have by applying Lemma \ref{l:phi-et-beta} again, 
\begin{align*}
 \Vert\sum_{2^jt\le 1}\vphi_j(H)e^{-itH}f\Vert_{p'} \le c \la t2^{j_t}\ra^\beta \Vert {\eta} (2^{-j}H) f \Vert_{p'}\le c \Vert f \Vert_{p'}
\end{align*}
where we take $\eta\in C^\iy_0$ with $\eta(x)\equiv 1$ on $[-1,1]$ 
so that $\eta(2^{-j_t}H)\sum_{2^jt\le1}\vphi_j(H)=\sum_{2^jt\le 1}\vphi_j(H)$. 
Therefore we obtain that if  $0<t\le1$,
\begin{align*}
\Vert e^{-itH}f \Vert_{p'}\lesssim
\Vert f\Vert_{p'}  +  t^\beta \Vert f \Vert_{B^{\beta,1}_{p'}(H)} \,, \end{align*}
which proves part (a).  

\vs{.071in}
\nd
(b) Inequality (\ref{e:u-B-Lp}) holds for $t>1$  
in virtue of Lemma \ref{l:disp-RS} and the remarks below Theorem \ref{th:global-V}.
For $0<t \le 1$, (\ref{e:u-B-Lp}) follows from the Besov embedding $B^{2\beta,1}_{p}(H)\hr B^{\beta,1}_{p'}(H)$,
which is valid because of the condition $|\pa_x^\al V(x)|\le c_\al$, $|\al|\le 2n$; cf. e.g. \cite{Tr83,Z06a}. 
\hB

\begin{remark}  
It seems from the proof that the smoothness order $2\beta$ in (\ref{e:u-B-Lp}) is optimal for the initial data $f$.
\end{remark}

\begin{remark}\label{r:u-B-2-Lp'} If working a little harder, we can show that
\begin{equation}\label{e:u-B-2-Lp}
\Vert  e^{-itH}f\Vert_{L^{p'}} \lesssim \la t \ra^{-n(\frac{1}{p}-\frac{1}{2})} \Vert  f\Vert_{B^{2\beta,2}_p(H)} ,
\end{equation}
 if assuming the upper Gaussian bound for the gradient of heat kernel of $H$ in short time, in addition to the conditions in Theorem \ref{th:global-V} (a).
The proof of (\ref{e:u-B-2-Lp}) is based on the embedding $B_{p'}^{0,2}(H)\hr F_{p'}^{0,2}(H)= L^{p'}$, $p'\ge 2$ which follows from a deeper 
result by applying the gradient estimates for $e^{-tH}$; see \cite{Z06a} and \cite{CouS06}. 
\end{remark}

\begin{corollary}\label{c:e{-itH}-B-F} Let $ 1\le p\le2$, $\al\in\R$ and $\beta=\beta(p)$. Suppose $V$ satisfies the
same conditions as in Theorem \ref{th:global-V} (b). The following estimates hold. \n
a) If $1\le q\leq \infty$, then
\begin{equation}\label{e:e-B-B}
\Vert e^{-itH}f\Vert_{B_p^{\alpha,q}(H)}
\lesssim \langle t\rangle^{-n(\frac{1}{p}-\frac{1}{2})} \Vert f \Vert_{B_p^{\alpha+2\beta,q}(H)}\,.
\end{equation}
b) If $1\le q\le p$, then
\begin{equation}\label{e:e-B-F}
\Vert e^{-itH}f\Vert_{F_p^{\alpha,q}(H)}
\lesssim \langle t\rangle^{-n(\frac{1}{p}-\frac{1}{2})} \Vert f \Vert_{B_p^{\alpha+2\beta,q}(H)}\,.
\end{equation}
\end{corollary}
\begin{proof} Substituting $\vphi_j(H)f$ for $f$ in (\ref{e:u-B-Lp}) we obtain
\begin{align*}
&\Vert  \vphi_j(H)e^{-itH}f\Vert_{L^{p'}} \lesssim \la t \ra^{-n(\frac{1}{p}-\frac{1}{2})} \Vert  \vphi_j(H)f\Vert_{B^{2\beta,1}_p(H)}\\
\approx& \la t \ra^{-n(\frac{1}{p}-\frac{1}{2})}\, 2^{2\beta j}\Vert \vphi_j(H)f\Vert_{L^p}
\end{align*}
where we used $\Vert \vphi_j(H)g\Vert_p\le c\Vert g\Vert_p$ by applying Lemma \ref{l:phi-et-beta} with $\theta=2^{-j}$ and $t=0$. 
Now multiplying $2^{j\al}$ and taking $\ell^q$ norms in the above inequality gives (\ref{e:e-B-B}). 
The estimate in (\ref{e:e-B-F}) follows from the embedding $B_p^{\alpha,q}(H)\hr F_p^{\alpha,q}(H)$ if $q\le p$, according to (\ref{e:B-F-B-emb}). 
\end{proof}


\begin{thebibliography}{99}






\bibitem{BS93}
M.~Beals, W.~Strauss, $L^p$ estimates for the wave equation with a potential.
{\em Comm. P.D.E.} {\bf 18} (1993), no. 7-8, 1365--1397.





 
 












\bibitem{CouS06} T.~Coulhon, A.~Sikora, 
Gaussian heat kernel upper bounds via Phragm\'en-Lindel\"of theorem. http://xxx.lanl.gov/abs/math/0609429, (2006).






\bibitem{DP05} P.~D'Ancona, V.~Pierfelice, 
On the wave equation with a large rough potential. {\em J. Funct. Anal}. {\bf 227} (2005), no. 1, 30--77.






















































\bibitem{JN94} A.~Jensen, S.~Nakamura, Mapping properties of functions
of Schr\"odinger operators between $L^p$ spaces and Besov spaces, in {\em Spectral and Scattering Theory and Applications}, Advanced 
Studies in Pure Math. {\bf 23} (1994), 187--209.



\bibitem{JSS} J.-L.~Journ\'e, A.~Soffer, C.~Sogge, Decay estimates for
Schr\"odinger operators. {\em Comm. Pure Appl. Math.}, vol. XLIV (1991), 573--604.


























\bibitem{OZ06}
G.~\'Olafsson, S.~Zheng, Function spaces associated with 
Schr\"odinger operators: the P\"oschl-Teller potential. {\em J. Fourier Anal. Appl}.  {\bf 12} (2006), no.6, 653--674. 







\bibitem{RS04} I.~Rodnianski, W.~Schlag, Time decay for solutions of Schr\"odinger equations
with rough and time-dependent potentials. {\em Invent. Math}. {\bf 155} (2004), no.3, 451--513.




\bibitem{Sch05a} 
W.~Schlag, Dispersive estimates for Schr\"odinger operators: A survey. 
{ http://lanl.arXiv.org/math.AP/0501037}, (2005). 







\bibitem{Si82}
B.~Simon, Schr\"odinger semigroups, {\em Bull. Amer. Math. Soc}. {\bf 7}
(1982), no.3, 447--526.
























\bibitem{Tr83} H.~Triebel, {\em Theory of Function Spaces},
Birkh\"{a}user Verlag, 1983.





\bibitem{V06a}
G.~Vodev,  Dispersive estimates of solutions to the Schr\"odinger equation in dimensions $n\geq 4$. 
{\em Asymptot. Anal}. {\bf 49} (2006), no. 1-2, 61--86. 




















\bibitem{Y95}
K.~Yajima, The $W^{k,p}$-continuity of wave operators for Schr\"odinger 
operators. {\em J. Math. Soc. Japan} {\bf 47} (1995), 551--581.





 


\bibitem{Z06a} 
S.~Zheng,  Littlewood-Paley theorem for  Schr\"odinger operators. 
{\em Anal. Theory. Appl}. {\bf 22} (2006), no.4, 353--361.



\bibitem{Z07}\bysame, {Spectral calculus, function spaces and dispersive equation with a critical potential}. In preparation.

\end{thebibliography}
\end{document}